\documentclass[12pt,a4paper,draft]{amsart}

\usepackage{color}
\usepackage[normalem]{ulem}
\title{Approximation by mappings with singular Hessian minors}
\author{Zhuomin Liu, Jan Mal\'y and Mohammad Reza Pakzad}
\thanks{This project is based upon work supported by, among others, the National Science
Foundation.
The research of
ZL has been supported by the Academy of Finland
via the center of excellence in analysis and dynamic research (project NO.271983).
JM was supported by the grant GA\,\v{C}R P201/15-08218S of the Czech Science Foundation.
MRP was partially supported by the NSF grant
DMS-1210258. 
}

\address
{ Department of Mathematics and Statistics, P.O.Box 35 (MaD)
FI-40014,
University of Jyv\"askyl\"a, Jyv\"askyl\"a, Finland}
\email{liuzhuomin@hotmail.com}
\address{Department of Mathematical Analysis, Faculty of Mathematics and Physics,
Charles University in Prague, So\-ko\-lovsk\'a 83,  Prague 8, 186\,75 Czech Republic}
\email{maly@karlin.mff.cuni.cz}
\address{Department of Mathematics, The Dietrich School of
Arts and Sciences,
University of Pittsburgh,
301 Thackeray Hall
Pittsburgh, PA 15260,
USA}
\email{pakzad@pitt.edu}

\usepackage{verbatim}

\theoremstyle{plain}
\newtheorem{thm}{Theorem}[section]

\newtheorem{lemma}[thm]{Lemma}
\newtheorem{corollary}[thm]{Corollary}

\newtoks\prt

\theoremstyle{definition}

\newtheorem{remark}[thm]{Remark}

\def\eqn#1$$#2$${\begin{equation}\label#1#2\end{equation}}

\numberwithin{equation}{section}

\def\ck#1{[#1]_k}
\def\A{\mathbf{A}}
\def\B{\mathbf{B}}
\def\M{\mathbf{M}}
\def\C{\mathcal{C}}

\def\F{\mathcal{F}}
\def\U{\mathcal{U}}

\def\diam{\operatorname{diam}}

\def\ep{\varepsilon}
\def\en{\mathbb N}
\def\er{\mathbb R}
\def\ff{\varphi}

\def\rank{\operatorname{rank}}
\def\sym{\mathrm{sym}}

\def\Q{\mathcal{Q}}
\def\P{\mathcal{P}}
\def\rn{\mathbb R^n}

\begin{document}

\begin{abstract}
Let $\Omega\subset\rn$ be a Lipschitz domain.
Given $1\leq p<k\leq n$ and any $u\in W^{2,p}(\Omega)$  belonging 
to the little H\"older class $c^{1,\alpha}$, we construct a sequence $u_j$ in the same space with $\operatorname{rank}D^2u_j<k$ almost everywhere such that $u_j\to u$ in $C^{1,\alpha}$ and weakly in $W^{2,p}$. This result is in strong contrast with  known regularity behavior of functions in $W^{2,p}$, $p\geq k$, satisfying the same rank inequality.
\end{abstract}
\maketitle

\section{Introduction}
The degenerate Monge-Amp\`ere equation received much attention due to its relation to the Gaussian curvature of flat surfaces, in particular, the image of a flat domain under isometric embedding. There are tremendous literature on the regularity theory of solutions to the Monge-Amp\`ere equation in various sense. They are too massive to survey here so we only mention a few. Regularity theory on convex solutions in the sense of Aleksandrov (on convex domains) was established in \cite{CNS}, \cite{GTW}; theory on nonconvex solution in viscosity sense to degenerate second order fully nonlinear equations can be found in \cite{IL}, \cite{Imb}; and nonconvex solutions to the degenerate Monge-Amp\`ere equation in the sense of Monge-Amp\`ere functions were developed in \cite{Fu}, \cite{Jerr}. In this paper, we mainly focus on pointwise solutions in the Sobolev spaces. A Sobolev function $f\in W^{2,p}(\Omega),p\geq 1$ where $\Omega$ is a Lipschitz domain in $\rn$, is a solution to the degenerate Monge-Amp\`ere equation in the a.e. sense if 
\eqn{MA}
$$
\det D^2 f(x)=0\qquad {\rm a.e.}\quad x\in\Omega.
$$ 
If $p\geq n$, then $\det D^2 f$ vanishes in the $L^1$ sense and is easily seen to be a solution to the degenerate Monge-Amp\`ere equation in the sense of \cite{Fu} and \cite{Jerr}. If we further assume a priori that $u$ is convex, then $\det D^2 f$ vanishing in $L^1$ also implies it is a solution to the degenerate Monge-Amp\`ere in the sense of Aleksandrov, see \cite{LM1} for discussion. Therefore, some regularity should be expected under sufficient integrability. In fact, Kirchheim \cite{Kirch} first proved that a $W^{2,\infty}$ solution to \eqref{MA} on a $2$-d domain  is developable, i.e., for each point in the domain, either there is a neighborhood of it on which $\nabla f$ is constant (and $f$ affine), or there is a line segment passing through this point and touching the boundary at both ends on which $\nabla f$ is constant (and $f$ affine). Pakzad \cite{pakzad} proved the same developability property for $W^{2,2}$ solutions to \eqref{MA} on a $2$-d domain. This in particular means that if $\Omega$ is a $2$-d domain and there is a sequence of $W^{2,2}$ solutions $f_j$ to \eqref{MA} which converges to some function $f$ in $C^1(\Omega)$, then necessary $f$ is developable since the limit of affine segments must be affine segments and if a $C^1$ map is affine on two intersecting line segment, it must be affine on their convex hull, see \cite{LP} for detail argument. 

In all dimension $n\geq 2$, the degenerate Monge-Amp\`ere equation can also be written as the rank inequality
\eqn{ri}
$$
\rank D^2f(x)<n\qquad {\rm a.e.}\quad x\in\Omega.
$$
More generally, we can consider functions $f\in W^{2,p}$ whose $k\times k$ minors of $D^2f$ vanishes in the a.e.\ sense, namely,
\eqn{rik}
$$
\rank D^2f(x)<k\qquad {\rm a.e.}\quad x\in\Omega,\qquad k\in \{2,3,...,n\},
$$
Again, if $p\geq k$, then all $k\times k$ minors vanishes in the $L^1$ sense and some regularity should be expected. Indeed, Jerrard and Pakzad \cite{jerr-pak} proved that solution to \eqref{rik} is weakly $(n-k+1)$ developable.
For precise definition of this property see \cite{jerr-pak}, in particular it implies that 
passing through each point in the domain there is a relatively open portion of a $(k-1)$-dimensional hyperplane
on which $\nabla f$ is $\mathcal{H}^{k-1}$-a.e.\ constant (and $f$ affine).
It might also happen that a point belongs to two $(k-1)$-dimensional hyperplanes, along each $f$ is affine, but the corresponding values of $\nabla f$ do not coincide. If $p\geq \{2k-2,n\}$, then  the solution is $C^1$-regular and hence developable as in the $2$-d case, with line segment replaced by hyperplanes. In particular, if there is a sequence of $W^{2,p},p\geq \{2k-2,n\}$ solution to \eqref{rik} converging to some function $f$ in $C^1(\Omega)$, then $f$ is developable. On the contrary, our main result Theorem \ref{t:main2} below shows that all the above regularity falls apart once $p<k$. 

The developability of isometric immersions of Sobolev regularity  is a related problem. Developability results for such $W^{2,p}$ immersions from $n$ dimensional domains into $\mathbb R^{n+k-1}$ were established as well in \cite{jerr-pak} for $p\geq \{2k-2,n\}$. It has been  shown  moreover in \cite{LM2} that each coordinate function of  a $W^{2,p}$, $p\geq 1$, isometric embedding of co-dimension  $k{-}1$ is a solution to \eqref{rik}. 
It is not known whether the type of constructions employed in this paper can also be  used  to create similar counter-examples for the isometric immersions.  On the other aspect, $C^{1,\alpha}$ approximation arise in recent years due to the successful construction by Conti, De Lellis and Sz\'ekelyhidi \cite{CDS}, of uniform approximation of any $C^1$ short map (see \cite{CDS} for definition) of co-dimension one by isometric immersions of class $C^{1,\alpha}$ for $\alpha<1/(n^2+n+1)$. 
When $n=2$, the limiting $\alpha$ has been subsequently improved to 1/5 recently by \cite{DIS}. 
Overall, our main results below, which also address the the H\"older regularity, seem to suggest that solutions to Monge-Amp\`ere equations, as formulated in \eqref{MA}, are more flexible than isometric immersions. This fact, however, must be contrasted with the results on 
the approximation property (for $\alpha <1/7$) or developability  (for $\alpha>2/3$) of $C^{1,\alpha}$  solutions  to the very weak formulation of the Monge-Amp\`ere equation in two dimensions:
\eqn{MA-vw}
$$
{\mathcal D}et\, D^2 f(x):= -\frac 12\, \mathrm{curl}\,\mathrm{curl}\, (\nabla f \otimes \nabla f)=0, 
$$ as recently discussed in \cite{lew-pak}.

\begin{thm}\label{t:main2}
Let $k\in\{2,3,\dots,n\}$ and $1\le p<k$.
Let $\Omega\subset\er^n$ be an open set with Lipschitz boundary and
$u\in W^{2,p}(\Omega)$. Then there
exists a sequence $u_j$ of functions in the same space
with $\rank D^2u_j <k$ a.e. such that
$u_j\to u$  weakly in $W^{2,p}$. If, moreover, $u$ is also in $c^{1,\alpha}(\overline\Omega)$ for $0<\alpha<1$, then 
the sequence $u_j$ can also be made to be in $c^{1,\alpha}(\overline\Omega)$ and converging  in $c^{1,\alpha}$ to $u$.  
\end{thm}

\begin{remark}
The method of the construction is adapted from \cite{LM1} where just one $W^{2,p}$
and $C^{1,\alpha}$ strictly convex solution of $\rank Du<k$ ($p<k$) has been constructed.
As pointed out in \cite[Remark 3]{LM1}, this convex solution can be close to a given
uniformly strictly convex smooth function. In the present paper, we approximate 
arbitrary functions of the assumed regularity, not only convex ones; on the other
hand, we do not care of preservation of convexity.

A similar convex example to \cite{LM1}, with even better (borderline) Sobolev regularity,
has been obtained also independently by Faraco, Mora-Corral and Oliva \cite{FMO}.

{Such pathological convex functions cannot solve the degenerate\penalty-10000 
\text{Monge-Amp\`ere} equation in the sense of Alexandrov.} 
On the other hand, even the convex solutions in the sense of Alexandrov to 
$$
\det D^2u=\rho\leq 1\quad {\rm in}\quad\Omega
$$
may fail to be $W^{2,1}$ as shown by Mooney \cite{Moo}. 
Note that in his example the support of $\rho$ is highly irregular. 
\end{remark}

\begin{remark}
The $k\times k$ minors of $\nabla^2 u_j$ in Theorem \ref{t:main2} vanish in the pointwise sense.
It is often convenient to interpret gradient or Hessian minors as distributions. Under this view, the minors
of $\nabla^2 u_j$ do not vanish identically, but turn to be signed singular measures, see Section~\ref{s:heme}.
\end{remark}
\begin{remark}
The main argument is that we can produce the desired approximation for any $C^2$-smooth function. The full generality follows from the observation that smooth functions are dense in Sobolev spaces and little H\"older spaces.
\end{remark}

Due to its null-Lagrangian behavior (cf.\ Remark \ref{null}), the Jacobian has been important in the theory of weak lower semicontinuity of integral functional, see \cite{Dac1} for a comprehensive introduction. Therefore, denseness of 
$\{f\in W^{1,p}\colon p<n,\;\det Df=0\;{\rm a.e.}\}$ in the weak topology $W^{1,p},p<n$ and its variants have also been an object of interest. For example, Haj{\l}asz \cite{Haj} proved the weak denseness of the class 
$\{f\in W^{1,p}\colon p<k,\;\operatorname{rank}Df<k\;{\rm a.e.}\}$ in the space $W^{1,p}$. Rindler \cite{Rin}, Koumatos, Rindler, and Wiedemann \cite{KRW}, \cite{KRW1} proved the weak denseness of $\{f\in W^{1,p}\colon 1<p<n,\; J_1\leq \det Df\leq J_2\;{\rm a.e.}\}$ in the space $W^{1,p}$ for any prescribed measurable function $J_1:\Omega\to [-\infty,\infty),J_2:\Omega\to (-\infty,\infty]$ with $J_1\leq J_2$ a.e. Our main theorem \ref{t:main2} can be adapted without difficulty into the following first order case. 
\begin{thm}\label{t:main1}
Let $k\in\{2,3,\dots,n\}$ and $1\le p<k$.
Let $\Omega\subset\er^n$ be an open set with Lipschitz boundary and
$u\in W^{1,p}$$(\Omega)$. Then there
exists a sequence $u_j$ of mappings in the same space with $\rank Du_j <k$ a.e.\
such that
$u_j\to u$  weakly in $W^{1,p}$. 
If, moreover, $u$ is also in $c^{0,\alpha}(\overline \Omega)$ for $0<\alpha<1$, then the sequence $u_j$ can also made to be in $c^{0,\alpha}(\overline\Omega)$ and converging  in $c^{0,\alpha}$ to $u$.

\end{thm}
All these results are in contrast to the weak continuity result by M{\"u}ller, Qi and Yan \cite{MQY}, which states that if $u_j\to u$ weakly in $W^{1,n}$ and $\det Du_j\geq 0$ a.e., then $\det Df_j\to \det Du$ weakly in $L_{\rm loc}^1$. 

We only demonstrate the construct for Theorem \ref{t:main2}, the proof of Theorem \ref{t:main1}, based on Corollary \ref{c:main} instead of Lemma \ref{l:main} is exactly the same with obvious modifications whenever needed.

\section{Preliminaries}

\subsection{Singular values}
Let $\A$ be a real $n\times n$ matrix. The \textit{singular values} 
of $\A$, i.e., the eigenvalues of $\sqrt{\A^T \A}$, which are real and nonnegative, will be denoted by 
$\mu_i=\mu_i(\A)$
indexed in
nondecreasing order, counting multiplicity, i.e., 
$$
\mu_1\leq
\mu_2\leq \cdots\leq \mu_n.
$$
We are interested in the expressions
$$
\ck{\A}=S_k(\mu_1,\dots,\mu_n),\qquad k=1,\dots,n,
$$
where $S_k$ is the $k$-th elementary symmetric function on $\rn$
$$
S_k(\mu_1,\dots,\mu_n)=\sum_{i_1<i_2<\dots<i_k}
\mu_{i_1}\cdots\mu_{i_k}.
$$
The expressions $\ck{\A}$ are important invariants as the order of singular values
may be difficult to watch. Since $A\to \mu_i(A)$ is continuous, it follows that 
$\ck{\A}$ depends continuously on $\A$. Moreover, denoting $\|\cdot\|$ the operator norm, 
\eqn{norm}
$$
\mu_n=\|\A\|
$$
and 
$$
\aligned
\rank \A<k&\iff \mu_1=\dots=\mu_{n-k+1}=0\iff \mu_{n-k+1}=0 
\\&
\iff
\mu_{n-k+1}\cdots\mu_n=0.
\endaligned
$$
Since 
$$
\mu_{n-k+1}\cdots\mu_n\le S_k(\mu_1,\dots,\mu_n)
\le \binom nk\; \mu_{n-k+1}\cdots\mu_n,
$$
we infer that
$$
\rank \A<k\iff \ck{\A}=0.
$$
Further, if the matrix $\A$ is symmetric, we write
$$
L_k(\A)=S(\lambda_1,\dots,\lambda_n),
$$
where $\lambda_1,\dots,\lambda_n$ are the eigenvalues of $\A$ (not put into absolute values). 
Note that (up to a permutation), $\mu_i=|\lambda_i|$, and thus
$|L_k(\A)|\le \ck{\A}$.

\subsection{Little H\"older spaces}
Let $\Omega\subset\rn$ be a Lipschitz domain. 
The \textit{little H\"older space} $c^{1,\alpha}(\overline\Omega)$ is the family of all 
continuously differentiable functions $u$ on $\overline\Omega$ with the property
$$
\lim_{r\to 0}\omega(Du,r^{\alpha})=0,
$$
where $\omega(f,\cdot)$ is the modulus of continuity
$$
\omega(f,r)=\sup\Bigl\{\frac{|f(y)-f(x)|}{|y-x|^{\alpha}}\colon x,y,\in\Omega,\;0<|y-x|<r\Bigr\}.
$$
The derivative $Du$, being uniformly continuous, can be naturally extended so that both $u$
and $Du$ are defined on $\overline\Omega$.
Clearly, $c^{1,\alpha}(\overline\Omega)$ is a closed subspace of $C^{1,\alpha}(\overline\Omega)$.
In fact, it is the closure of $C^{\infty}(\overline\Omega)$ in $C^{1,\alpha}(\overline\Omega)$, see Section
\ref{s:concl} below.



\section{The building pattern}

The following lemma follows from \cite[Lemma 1 and (3.23)]{LM1}, taking into account of Remark 2 there. 
\begin{lemma}\label{l:main}
Let $k\in\{2,\dots,n\}$, $1\le p<k$ and $\ep_0>0$.
Let $\A$ be a symmetric mapping
and $Q\subset\rn$ be a closed cube. 
Then there exist $\tau\in (0,1)$
depending only on $n$, $k$ and $p$ and
a function $g\in \C^2(\rn)$ with support in $Q$ such that 
such that 
\eqn{prop0}
$$
|g(x)|+|\nabla g(x)|\le\ep_0,\qquad x\in Q,
$$
\eqn{prop1}
$$
\|\nabla^2 g(x)\|\le\|\A\|,\qquad x\in Q,
$$
\eqn{prop2}
$$
\|\nabla^2 g(x)\|^k\le \ck{\A},\qquad x\in Q,
$$
\eqn{prop3} 
$$
\int_Q \ck{\A+\nabla^2 g(x)}^{p/k}\,dx\le \tau|Q|\ck{\A}^{p/k},
$$
and
\eqn{prop4}
$$
\bigg|\int_Q \ck{\A+\nabla^2 g(x)}-\ck{\A}\,dx\biggr| \le \ep_0|Q|,
$$
\end{lemma}

By a well-known result of polar decomposition, see \cite[Section 3.2, theorem 2]{EG}, we can write $\B=O\circ \A$ where $\A=\sqrt{\B^T \B}$ and $O:\rn\to \rn$ is an orthogonal map. We then apply Lemma \ref{l:main} to $\A$ and define
$$
h:=O\circ \nabla g.
$$
Therefore, $h\in \C^1(\rn, \rn)$ with support in $Q$ and 
$$
\nabla h =O \nabla^2 g.
$$
Moreover, 
$$
(\B+\nabla h)^T (\B+\nabla h)=(\A+\nabla^2 g)^T(\A+\nabla^2 g)=(\A+\nabla^2 g)^2
$$
Hence, 
$$
\ck{\B+\nabla h}=\ck{\A+\nabla^2 g}.
$$
With the above remark, it follows immediate from \ref{l:main},
\begin{corollary}\label{c:main}
Let $k\in\{2,\dots,n\}$, $1\le p<k$ and $\ep_0>0$.
Let $\B$ be a linear mapping
and $Q\subset\rn$ be a closed cube. 
Then there exist $\tau\in (0,1)$
depending only on $n$, $k$ and $p$ and
a function $h\in \C^1(\rn, \rn)$ with support in $Q$ such that 
such that 
\eqn{cprop01}
$$
|h(x)|\le\ep_0,\qquad x\in Q,
$$
\eqn{prop11}
$$
\|\nabla^2 h(x)\|\le\|\B\|,\qquad x\in Q,
$$
\eqn{prop21}
$$
\|\nabla h(x)\|^k\le \ck{\B},\qquad x\in Q,
$$
\eqn{prop31} 
$$
\int_Q \ck{\B+\nabla h(x)}^{p/k}\,dx\le \tau|Q|\ck{\B}^{p/k},
$$
and
\eqn{prop41}
$$
\bigg|\int_Q \ck{\B+\nabla h(x)}-\ck{\B}\,dx\biggr| \le \ep_0|Q|,
$$
\end{corollary}

\section{The construction from a smooth initial function}

After some thoughts, which we postpone to Section \ref{s:concl},
the proof of Theorem \ref{t:main2} reduces to
the following lemma.

\begin{lemma}\label{l:main1} 
Let $Q_0=[0,1]^n$, $w\in C^2(Q_0)$, $1<p<k$ and $\ep>0$. Then there exists
$f\in W^{2,p}((0,1)^n)\cap c^{1,\alpha}(Q_0)$ such that $\rank \nabla^2f<k$ a.e.\
in $Q_0$,
$f=w$ and $\nabla f=\nabla w$
on $\partial Q_0$, $\|f\|_{2,p}\le C(\|w\|_{2,p}+1)$ (with $C=C(n,p,k)$) and 
$\|f-w\|_{C^{1,\alpha}}<\ep$.
\end{lemma}

\begin{remark}\label{null}
For the case $p\geq k$, the null-Lagrangian behavior implies that the integral of the $k\times k$ minors of $D^2f$ must be the same as those of $D^2w$. Our result shows that the null-Lagrangian effect fails
when $p<k$.
\end{remark}

\subsection{The plan of the construction}\label{s:plan}
From now on, $C$ will denote a general constant that possibly depends on $w, n, k, \alpha$ and varies line by line.
Consider the partitions
\eqn{partition}
$$
\Q_j:=\biggl\{
\Bigl[\frac{z_1-1}{m_j},\;\frac{z_1}{m_j}\Bigr]\times\dots\times 
\Bigl[\frac{z_n-1}{m_j},\;\frac{z_n}{m_j}\Bigr]:\;z\in \{0,\dots,m_j\}^n
\biggr\},
$$
where integers $m_j$ will be specified later.
Set $f_0=w$.
Let $\tau\in (0,1)$ be from Lemma \ref{l:main}. Let $\ep_j\in (0, \tau^j/3)$ be chosen later.
We claim that we can use induction to
construct a sequence $(g_j)_j$  of $C^2$ mappings on
$\rn$ such that, for each $j=1,2,\dots$ and $Q\in\Q_j$, both $g_j$ and $\nabla g_j$ vanish on $\partial Q_j$, and
the functions $g_j$ and 
\eqn{howfj}
$$
f_j:=f_0+g_1+\dots+g_j,\qquad j=1,2\dots
$$
satisfy
\eqn{cprop0}
$$
|g_j(x)|+|\nabla g_j(x)|\le\ep_j,\qquad x\in Q_0,
$$
\eqn{cprop2}
$$
\|\nabla^2 g_j(x)\|^p\le  
\ck{\nabla^2 f_{j-1}(x)}^{p/k}+\tau^{j},\qquad x\in Q_0,
$$
\eqn{cprop3} 
$$
\int_{Q_0} \ck{\nabla^2 f_j(x)}^{p/k}\,dx\le \tau
\int_{Q_0} \ck{\nabla^2 f_{j-1}(x)}^{p/k}\,dx+\tau^{j},
$$
and 
\eqn{cprop4}
$$
\biggl|\int_{Q} \ck{\nabla^2 f_j(x)}-\ck{\nabla^2 f_{j-1}(x)}\,dx\biggr|\le \tau^j|Q|,\qquad Q\in\Q_j.
$$
If we succeed in constructing the sequence $f_j$, we use \eqref{cprop0}
to establish existence of the limit function
$$
f=\lim_{j\to\infty}f_j.
$$

\subsection{Details of the construction}
Throughout we use uniform continuity of $\nabla^2f_{j-1}$ on $Q_0$. 
Let 
$$
K_j=\sup\{\|\nabla^2f_{j-1}\|:\;x\in Q_0\}.
$$
Since continuous functions on $\er_{\sym}^{n\times n}$
are locally uniformly continuous, there exists $\delta_j>0$ such that
for each $\M',\M\in\er_{\sym}^{n\times n}$ we have
\eqn{deltaj1}
$$
\left.
\aligned
&\|\M\|\le 2K_j,\\
&\|\M'-\M\|<\delta_j
\endaligned\right\}
\!\!\implies \!\!
\left\{
\aligned
&|\ck{\M'}^{p/k}-\ck{\M}^{p/k}|<\tfrac 12 \tau^j,
 \textrm{  and}\\
&|\ck{\M'}-\ck{\M}|<\tfrac 13\tau^2
\endaligned
\right.
$$
Then we find $\beta_j>0$ such that for each $x,y\in Q_0$ we have
\eqn{betaj}
$$
|y-x|<\beta_j\implies \|\nabla^2 f_{j-1}(y)-\nabla^2 f_{j-1}(x)\|<\delta_j.
$$
Find $m_j\in\en$ such that $m_j$ is a multiple of $m_{j-1}$ (if $j\ge 2$),
\eqn{geom}
$$
m_j\ge 2^j,
$$
and
\eqn{diam}
$$
\frac{\sqrt n}{m_j}<\min\bigl\{\tfrac12\ep_j,\;\beta_j\bigr\},
$$
and consider the partition \eqref{partition}.
For each $Q\in\Q_j$ let $x_{Q}$ be the center of $Q$ and set
$$
\A_Q:=\nabla^2 f_{j-1}(x_Q).
$$
We use Lemma \ref{l:main} to find a function $g_Q\in\C^2(\rn)$ with support in
$Q$ such that $g=g_Q$ satisfies \eqref{prop0}--\eqref{prop4} with
$\ep_0=\ep_j$ and $\A=\A_Q$.
Set
$$
g_j(x):=g_Q(x),\quad \textrm{for } x\in Q\in \Q_j.
$$
Then \eqref{cprop0} follows directly from \eqref{prop0}.

\subsection{Property \eqref{cprop2}}
If $x\in Q\in\Q_j$, then
$$
|x-x_Q|\le \diam Q\le \beta_j
$$
and thus
\eqn{osc}
$$
\|\nabla^2 f_{j-1}(x)-\A_Q\|<\delta_j. 
$$
By \eqref{deltaj1},
$$
\aligned
\ck{\A_Q}^{p/k}&\le \ck{\nabla^2f_{j-1}(x)}^{p/k}+|\ck{\nabla^2f_{j-1}(x)}^{p/k}-\ck{\A_Q}^{p/k}|
\\&\le  \ck{\nabla^2f_{j-1}(x)}^{p/k}+\tfrac 12 \tau^j.
\endaligned
$$
Using \eqref{prop2} we infer that 
$$
\|\nabla^2g_j(x)\|^p \le \ck{\A_Q}^{p/k}
\le \ck{\nabla^2f_{j-1}(x)}^{p/k}+\tau^{j}. 
$$

\subsection{Property \eqref{cprop3}}
If $x\in Q\in\Q_j$, then by \eqref{prop1} 
$$
\|\nabla^2g_j(x)\|\le 
\|\A_Q\|\le K_j,
$$
thus
$$
\|\nabla^2f_j(x)\|\le 2K_j.
$$
By \eqref{deltaj1} and \eqref{osc}, 
$$
\aligned
&|\ck{\nabla^2 f_{j-1}(x)}^{p/k}-\ck{\A_{Q}}^{p/k}|\le \tfrac 12 \tau^j, \textrm{  and,}\\
&|\ck{\nabla^2 f_{j}(x)}^{p/k}-\ck{\A_{Q}+\nabla^2g_j(x)}^{p/k}|<\tfrac 12 \tau^j.
\endaligned
$$
Integrating with respect to $x\in Q$  we obtain
\eqn{fin1}
$$
\aligned
\int_Q\ck{\nabla^2 f_{j}(x)}^{p/k}\,dx&\le 
\int_Q\ck{\A_{Q}{+}\nabla^2g_j(x)}^{p/k}\,dx
\\&\qquad +\tfrac 12 \tau^j|Q|
\endaligned
$$
and
\eqn{fin2}
$$
\int_Q\ck{\nabla^2f_{j-1}(x)}^{p/k}\,dx
\ge \ck{\A_{Q}}^{p/k}|Q|
-\tfrac 12 \tau^j|Q|
$$
From \eqref{fin1}, \eqref{prop3} and \eqref{fin2} we obtain
\eqn{fin3}
$$
\aligned
&\int_Q\ck{\nabla^2 f_{j}(x)}^{p/k}\,dx\le 
\int_Q\ck{\A_{Q}{+}\nabla^2g_j(x)}^{p/k}\,dx
+|Q|\tfrac 12 \tau^j
\\&\le\tau\ck{\A_Q}^{p/k}\,|Q|
+\tfrac 12 \tau^j|Q|
\\&\quad\le \tau\int_Q\ck{\nabla^2f_{j-1}(x)}^{p/k}\,dx +\tau^j |Q|.
\endaligned
$$

\subsection{Property \eqref{cprop4}}
Similarly by \eqref{deltaj1} and \eqref{osc}, 
$$
\aligned
&|\ck{\nabla^2 f_{j-1}(x)}-\ck{\A_{Q}}|\le \tfrac 13 \tau^j, \textrm{  and,}\\
&|\ck{\nabla^2 f_{j}(x)}-\ck{\A_{Q}+\nabla^2g_j(x)}|<\tfrac 13 \tau^j.
\endaligned
$$
From \eqref{prop4} with $\ep_0$ now replaced by $\ep_j\in (0, \tau^j/3)$, we obtain that
\eqn{meas}
$$
\aligned
&\biggl|\int_Q\ck{\nabla^2 f_{j}(x)}-\ck{\nabla^2f_{j-1}(x)}\,dx\biggr|
\\&\le 
\biggl|\int_Q\ck{\A_{Q}{+}\nabla^2g_j(x)}-\ck{\nabla^2f_{j-1}(x)}\,dx\biggr|
+\tfrac 13 \tau^j|Q|
\\&\le \biggl|\int_Q\ck{\A_{Q}{+}\nabla^2g_j(x)}-\ck{\A_Q}\,dx\biggr|
\\&+ \biggl|\int_Q\ck{\A_Q}-\ck{\nabla^2f_{j-1}(x)}\,dx\biggr|
+\tfrac 13 \tau^j|Q|
\\&\le\ep_j|Q|+\tfrac 13 \tau^j|Q|+\tfrac 13 \tau^j|Q|\le \tau^j|Q|.
\endaligned
$$
Summing over $Q\in\Q_j$ we conclude \eqref{cprop4}.

\section{Sobolev estimate}
Since $\ck{\A}\le \binom nk\|\A\|^k$ holds for any $\A\in\er_{\sym}^{n\times n}$, applying \eqref{cprop3} iteratively, together with the finiteness of $\|\nabla^2 f_0\|_p$ we obtain
\eqn{iteration}
$$
\aligned
&\int_{Q_0} \ck{\nabla^2 f_j(x)}^{p/k}\,dx\le \tau^j
\int_{Q_0} \ck{\nabla^2 f_0(x)}^{p/k}\,dx+j\tau^j
\\&\le \tau^j\binom nk\|\nabla^2 f_0\|_p^p +j\tau^j\to 0 \quad \textrm{as }j\to \infty.
\endaligned
$$
By \eqref{cprop2},
\eqn{grad}
$$
\aligned
\int_{Q_0}|\nabla^2 g_j(x)|^p\,dx&\le
\int_{Q_0}\ck{\nabla^2 f_{j-1}(x)}^{p/k}\,dx+ \tau^j
\\&\le \tau^{j-1}\binom nk
\|\nabla^2 f_0\|_p^p\,dx+j\tau^{j-1}. 
\endaligned
$$
Note that 
$$
\sum_j j\tau^{j-1} <\infty
$$
and $\|\nabla^2 f_0\|_p$ is finite, thus $\sum_jg_j$ converges in $W_0^{2,p}(Q_0)$
(and the sum can be identified as $f-f_0$). It follows
that 
$$
f=\lim_{j\to\infty} f_j=f_0+\sum_{j=1}^{\infty}g_j\in W^{2,p}(Q_0),
$$
with
\eqn{gradb}
$$
\|\nabla^2 f\|_p^p\leq C_1 \|\nabla^2 f_0\|_p^p+C_2.
$$

\section{The rank property}\label{s:rank}

Since $\ck{\A}\le \binom nk\|\A\|^k$ holds for any $\A\in\er_{\sym}^{n\times n}$,
from \eqref{iteration} and \eqref{grad} we infer that
$$
\int_{Q_0}\ck{\nabla^2 f(x)}^{p/k}\,dx=\lim_{j\to\infty}
\int_{Q_0}\ck{\nabla^2 f_{j-1}(x)}^{p/k}\,dx=0
$$
(Alternatively, we could use Fatou's lemma). Therefore, $\ck{\nabla^2f}=0$
a.e., which means that $\rank\nabla^2f<k$ a.e.

\section{H\"older $C^{1,\alpha}$ estimate} 
Denote
$$
K_j=2\sup\{\|\nabla^2f_{j-1}(x)\|:\;x\in Q_0\},
$$
and from \eqref{prop1} 
$$
\|\nabla^2g_j(x)\|\le 
\|\A_Q\|\le \tfrac12\, K_j,\qquad x\in Q_0.
$$
On the other hand, from \eqref{cprop0}, 
$$
\|\nabla g_j(x)\|\le \ep_j,\qquad x\in Q_0.
$$
By interpolation,
$$
\|g_j\|_{1,\alpha}\le \|\nabla^2g_j\|_{\infty}^{\alpha}\|2\nabla g_j(x)\|_{\infty}^{1-\alpha}\le K_j^{\alpha}\ep_j^{1-\alpha}.
$$
Choose $\ep_j$ so that $K_j^{\alpha}\ep_j^{1-\alpha}<\ep(1-\tau)\tau^j$, 
$$
\|f-f_0\|_{1,\alpha}\le \frac{\ep}{1-\tau}\sum_j \tau^j \leq \ep. 
$$

Since $f_j$ are smooth, this implies $f\in c^{1,\alpha}(Q_0)$.

\section{The Hessian measure}\label{s:heme}

Let $f$ be the function from Lemma \ref{l:main1}.
In this section we will show that the (vanishing) pointwise $k\times k$ minors of $\nabla^2f$
are only singular parts of ``Hessian measures'', which carry some additional information,
and, in particular, explain why there is a failure of rigidity (or ``null Lagrangian property'')
for $f$.

We obtain the Hessian measure as a result of a relaxation procedure in spirit of 
Marcellini \cite{Mar}, Fonseca and Marcellini \cite{FMa}, Fonseca and Mal\'y \cite{FM}, 
\cite{FM1}
and Bouchitt\'e, Fonseca and Mal\'y \cite{BFM}.
Let us emphasize that
it is not aim to prove any general relaxation result, rather we want to have our reasoning 
as simple as possible. 
Therefore we require uniform convergence of gradient in the definition
of the convergence used for relaxation. See also Remark \ref{r:FM}.

The key observation is that the function $\ck{\cdot}$ is polyconvex,
see \cite[Theorem 5.39]{Dac1}.

Let $\Omega\subset\rn$ be an open set and $u\in W^{2,p}(\Omega)\cap C^1(\Omega)$.
For and open set $U\subset\Omega$ we define
$\U(u,U)$ as the family
of all sequences $(u_j)_j$ of $C^2$ functions which converge to $u$ 
in  $C^1(U)$ and weakly in $W^{2,p}(U)$, and 
$$
\F(u,U)=\inf\Bigl\{\liminf_j\int_{U}\ck{\nabla^2 u_j(x)}\,dx:\;
(u_j)_j\in \U(u,U)\Bigr\}.
$$
We say that $\boldsymbol\mu_k$ is the \textit{$k$-th positive Hessian measure}
of $u$ if for each open set $U\subset\subset\Omega$ we have
$$
\boldsymbol\mu_k(U)\le \F(u,U)\le \boldsymbol\mu_k(\overline U).
$$
Such a measure $\boldsymbol\mu_k$ is obviously unique if it exists.

We prove the following 

\begin{thm} Let $f$ be the function from  Lemma \ref{l:main1}, with $p>k-1$.
Then there exists
a $k$-th positive Hessian measure $\boldsymbol\mu_k$ of $f$. Moreover, 
$\boldsymbol\mu_k$ is obtained as the weak* limit in $C(Q_0)^*$ of the sequence
$\ck{\nabla^2 f_j}$ where $f_j$ are as in \eqref{howfj}.
\end{thm}


\begin{proof}
{\sc Step 1.}
The property \ref{cprop4} implies boundedness of $(\ck{\nabla^2 f_j})_j$ in $L^1$.
Namely,
summing this estimate over $Q\in \Q_j$, we obtain
$$
\|\,\ck{\nabla^2 f_j}\,\|_1\le K,\qquad j=1,2,\dots,
$$
where
$$
K=\|\,\ck{\nabla^2 f_0}\,\|_1+\tau+\tau^2+\dots.
$$
Choose $\ff\in C^1(Q_0)$.
 and $j\in\en$. 
For each $Q\in\Q_j$,
$$
\aligned
&\Bigl|\int_Q(\ck{\nabla^2f_j(x)}-\ck{\nabla^2f_{j-1}(x)})\ff(x)\,dx\Bigr|
\\&\qquad
\le \Bigl|\int_Q(\ck{\nabla^2f_j(x)}-\ck{\nabla^2f_{j-1}(x)})\ff(x_Q)\,dx\Bigr|
\\&\qquad\qquad
+ \Bigl|\int_Q(\ck{\nabla^2f_j(x)}-\ck{\nabla^2f_{j-1}(x)})(\ff(x)-\ff(x_Q))\,dx\Bigr|
\\&\qquad
\le \|\ff\|_{\infty}\tau^j|Q|
\\&\qquad\qquad
+\diam Q\|\nabla\ff\|_{\infty}\;\int_Q(\ck{\nabla^2f_j(x)}+\ck{\nabla^2f_{j-1}(x)})\,dx.
\endaligned
$$
Summing over $Q\in\Q_j$ and using \eqref{geom} we obtain
$$
\aligned
&\Bigl|\int_{Q_0}(\ck{\nabla^2f_j(x)}-\ck{\nabla^2f_{j-1}(x)})\ff(x)\,dx\Bigr|
\\&\qquad 
\le
\|\ff\|_{\infty}\tau^j+CK2^{-j}\|\nabla\ff\|_{\infty}.
\endaligned
$$
It follows that the sequence $(\ck{\nabla^2f_j})_j$ converges weak* in 
$C(Q_0)^*$. We set $\boldsymbol\mu_k$ to be this weak* limit.

{\sc Step 2.} Given and open set $U\subset Q_0$, we want to prove
\eqn{flem}
$$
\F(f,U)\le \boldsymbol\mu_k(\overline U).
$$
Let $W\subset \rn$ be an open set such that $\overline U\subset W$ and 
$\eta$ be a cut-off function with support in $W$ such that $0\le \eta\le 1$
in $W$ and $\eta=1$ on $\overline U$. Then 
$$
\boldsymbol\mu(\overline U)\le \lim_j\int_{Q_0}\eta\ck{\nabla^2 f_j(x)}\,dx=
\int_{Q_0}\eta\,d \boldsymbol\mu_k\le \boldsymbol\mu_k(W\cap Q_0).
$$
Varying $W$ we obtain \eqref{flem}.

{\sc Step 3. } Fix $l\in\en$ and $Q\in \Q_l$. Then by \eqref{deltaj1},
\eqn{lsc1} 
$$
\aligned
\int_Q\ck{\nabla^2f_{l-1}(x)}\,dx&\le \bigl(\ck{\A_Q}+\tfrac13\tau^l\bigr)\,|Q|.
\endaligned
$$
Consider the cube $P$ which is concentric with $Q$ and its edge is $(1-\gamma)$-multiple of the edge
of $Q$ with $0<\gamma<<1$. Then
\eqn{howP}
$$
|Q\setminus P|\le C\gamma|Q|.
$$
Find a smooth cut-off function $\eta$ with support in $Q$ and values in $[0,1]$
such that $\eta=1$ on $P$ and
$$
\gamma\|\nabla\eta\|_{\infty}+\gamma^2\|\nabla^2\eta\|_{\infty}\le C,
$$
where $C$ may depend on the size of $Q$. Let $(u_j)_j\in\U(u,U)$.
Let $\gamma>0$ and find $j_0\in\en$ such that 
\eqn{j0}
$$
|u_j-f|\le \gamma^2,\quad |\nabla u_j-\nabla f|\le \gamma,\qquad j\ge j_0.
$$
Set 
$$
a_Q(x)=f_{l-1}(x_Q)+\nabla f_{l-1}(x_Q)(x-x_Q)+\tfrac12 \A_Q(x-x_Q)\cdot(x-x_Q).
$$
and 
$$
v_j(x)=a_Q(x)+\eta(x)(u_j(x)-a_Q(x).
$$
Now, by polyconvexity (used as quasiconvexity) of $\ck{\cdot}$, we have
\eqn{lsc2}
$$
|Q|\ck{\A_Q}\le\int_Q\ck{\nabla^2v_j(x)}\,dx.
$$
We have 
$$
\nabla^2v_j(x)=\eta(x)\nabla^2u_j(x)+\xi(x),
$$
where
$$
|\xi(x)|\le C(1+\gamma^{-2}|u_j(x)-a_Q(x)|+
\gamma^{-1}|\nabla u_j(x)-\nabla a_Q(x)|)\le C
$$
with $C$ depending e.g.\ on $f_{l-1}$ and its derivatives but not on $j$.
The function $\ck{\cdot}$ satisfies
\begin{align}
&|\ck{\B}-\ck{\A}|\le C(|\A|^{k-1}+|\B|^{k-1})|\B-\A|,\label{nevim}
\\& \ck{t\A}\le\ck{\A},\quad \A,\B\in\er^{n\times n}, \;t\in[0,1],\nonumber
\end{align}
see \cite{Mar} for \eqref{nevim}.
Applying \eqref{nevim} to $\B=\eta\nabla^2u_k$ and $\A=\nabla^2v_k$ we obtain
\eqn{lsc3}
$$
\aligned
\int_Q\ck{\nabla^2v_j(x)}\,dx&\le \int_Q\ck{\nabla^2u_j(x)}\,dx
\\&\quad
+C\int_{Q\setminus P} (1+|\nabla^2u_j|^{k-1})\,dx
\endaligned
$$
and using the H\"older inequality we can continue
\eqn{lsc4}
$$
\int_{Q\setminus P} (1+|\nabla^2u_j|^{k-1})\,dx\le C\gamma^{\beta}\|\nabla^2 u_j\|_p
$$
with $\beta=\frac{p-k+1}{p}$.
Getting together \eqref{lsc1}--\eqref{lsc4} we conclude
\eqn{lsc5}
$$
\int_Q\ck{\nabla^2 f_{l-1}(x)}\,dx\le \int_Q\ck{\nabla^2u_j(x)}\,dx+
C\gamma^{\beta} + \tau^j|Q|.
$$

{\sc Step 4}. Let $U\subset Q_0$ be an open set and $t<\boldsymbol\mu_k(U)$. We find $l_0\in \en$
such that $\mu_k(U_{l_0})>t$, where $U_j$ denotes the union of all $Q\in \Q_j$ with 
$Q\subset U$. Further, we find $l\in \en$ such that $l\ge l_0$ and 
$\int_{U_l}\ck{\nabla^2f_{l-1}(x)}\,dx>t$. We find $\gamma$ such that 
$$
\int_{U_l}\ck{\nabla^2f_{l-1}(x)}\,dx>t+CN\gamma^{\beta}
$$
where $C$ is the constant from \eqref{lsc5} and $N$ is the number of cubes
in $\Q_l$. Let $j_0$ verifies \eqref{j0} for all cubes in $\Q_l$. By \eqref{lsc5},
$$
t\le \int_{U_l}\ck{\nabla^2 u_j(x)}\,dx+\tau^j,\qquad j\ge j_0,
$$
and thus, letting $t\to\boldsymbol\mu_k(U)$ we obtain 
$$
\boldsymbol\mu_k(U)\le \F(f,U),
$$
Comparing with \eqref{flem} we conclude that 
$\boldsymbol\mu_k(U)$ is the $k$-th positive Hessian measure
of $f$. 
\end{proof}

\begin{thm} Let $f$ be the function from  Lemma \ref{l:main1}, with $p>k-1$
and $\boldsymbol\mu_k$ be the $k$-th positive Hessian measure $\boldsymbol\mu_k$ of $f$.
Then $\boldsymbol\mu_k$ is singular.
\end{thm}

\begin{proof}

{\sc Step 1.} We claim that $\ck{\nabla^2f_j}\to 0$ a.e. Indeed,
by \eqref{iteration} we have
$$
\int_{Q_0}\Bigl(\sum_{j}\ck{\nabla^2f_j(x)}^{p/k}\Bigr)\,dx=
\sum_j\int_{Q_0}\ck{\nabla^2f_j(x)}^{p/k}\,dx<\infty,
$$
so that
$$
\sum_{j}\ck{\nabla^2f_j(x)}^{p/k}<\infty
$$
for a.e.\ $x\in Q_0$ and our claim easily follows.

{\sc Step 2.} Let $\theta$ be the Radon-Nikod\'ym derivative
of the (absolutely continuous part of )$\boldsymbol\mu_k$.
By the Lebesgue differentiation theory, for a.e.\ $x\in Q_0$
we have
\eqn{good1}
$$
\lim_{r\to 0+}\frac{1}{|B(x,r)}\int_{B(x,r)}|\theta(y)-\theta(x)|\,dy=0.
$$

{\sc Step 3.} Pick $x\in Q_0$ such that $\lim_j\ck{\nabla^2f_j(x)}=0$
and \eqref{good1} holds. Choose $\ep>0$
and find $l\in\en$ such that 
$$
\frac{2\tau^{l}}{1-\tau}<\ep,
$$
$$
\ck{\nabla^2f_{l-1}(x)}<\ep
$$
and for each $r<\frac{n}{m_{l}}$ 
($m_l$ are as in \eqref{partition})
we have
\eqn{gooda1}
$$
\frac{1}{|B(x,r)}\int_{B(x,r)}|\theta(y)-\theta(x)|\,dy<\ep.
$$
Find $Q\in Q_{l}$ such that $x\in Q$. By \eqref{osc}, \eqref{deltaj1}
we have
$$
\int_{Q}\ck{\nabla^2 f_{l-1}(y)}\,dy\le 
|Q|(\ck{\nabla^2 f_{l-1}(x)}+\tau_j)\le |Q|(\ep+\tau_j)
$$
and using \eqref{cprop4} we obtain 
$$
\int_{Q}\ck{\nabla^2 f_j(y)}\,dy\le 
|Q|(\ep+\tau^j+\tau^j+\tau^{j+1}+\tau^{j+2}\dots)
\le 2\ep|Q|,
$$
for each $j\ge l$.
Letting $j\to \infty$ we obtain 
$\boldsymbol\mu_k(Q\setminus\partial Q)<2\ep|Q|$.
It follows that 
$$
\aligned
\theta(x)&\le \frac1{|Q|}\int_Q\theta(y)\,dy+
\frac1{|Q|}\int_Q|\theta(y)-\theta(x)|\,dy
\\&\le
\frac1{|Q|} \boldsymbol\mu_k(Q\setminus\partial Q)
+\frac1{|Q|}\int_B|\theta(y)-\theta(x)|\,dy
\\&\le
2\ep
+\ep \frac{|B|}{|Q|}.
\endaligned
$$
Letting $\ep\to 0$ we obtain $\theta(x)=0$ and this holds for
a.e.\ $x\in Q_0$.
\end{proof}

\begin{remark}
\label{r:FM}
Using  \cite[Theorem 1.7]{FM1} by Fonseca and  Mal\'y we could conclude
that a similar relaxation procedure, using weak convergence in $W^{2,p}$ only,
leads to a singular measure when applied to our function $f$.  
However, we showed also that our positive Hessian measure is a weak* limit in $C(Q_0)^*$ of 
the sequence $\ck{\nabla^2 f_j}$, so by the null-Lagrangian property of Remark \ref{null}, it is nontrivial unless all integrals of $k\times k$-minors of 
$\nabla^2w$ vanish.
The relaxation measure of \cite{FM1} is also obtained as a limit
of evaluations of the functional on a minimizing sequence, but
under assumptions which are not satisfied in our example and, moreover,
there is no evidence that our sequence $(f_j)_j$ used
for the construction has the minimizing property. It could theoretically
happen that the true minimizing sequence gives a trivial relaxation measure.
  In conclusion, it is possible that the relaxation measure of \cite{FM1}
does the same job as our ``positive Hessian measure'', but if so, it cannot be 
verified by a mere reference to  \cite{FM1}.
\end{remark}

\section{Proof or Theorem \ref{t:main2}} \label{s:concl}

Let $\Omega\subset\rn$ be a (bounded) Lipschitz domain. We will  present the complete proof for the second case, i.e.\ when  $u\in W^{2,p}(\Omega)\cap c^{1,\alpha}(\Omega)$.   For the first case, i.e. when $u$ is merely in  $W^{2,p}(\Omega)$, the proof follows the same lines and is actually much shorter since there is no necessity to control the $C^1$ norms and the moduli of continuity of the approximating sequence. 

Our plan is to use Lemma \ref{l:main1} to obtain the desired approximation. To this end,
we first obtain a slightly larger $\Omega_j'$ with $\Omega\subset\subset\Omega_j'$ and 
$v_j\in W^{2,p}(\Omega_j')\cap c^{1,\alpha}(\Omega_j')$ such that the modulus of continuity
$\omega_j$ of $\nabla v_j$ is estimated by the modulus of continuity of $\nabla u$
in the sense of \eqref{modcon} below,
and
$$
\|v_j-u\|_{C^1(\overline\Omega)}<2^{-j},\qquad \|v_j-u\|_{W^{2,p}(\Omega)}\le 2^{-j}
$$ 
(the details how to find such an approximation are described below in Lemma \ref{l:lipappr}).
By mollification, we find open sets $\Omega_j$ and $w_j\in C^{\infty}(\overline\Omega_j)$ such that
$$
\Omega\subset\subset\Omega_j\subset\subset\Omega_j',
$$
the modulus of of continuity of $\nabla w_j$ is no more than $\omega_j$
and
$$
\|v_j-w_j\|_{C^1(\overline\Omega_j)}<2^{-j}, \qquad  \|w_j-v_j\|_{W^{2,p}(\Omega_j)}\le 2^{-j}.
$$
Now, we choose $q\in [p,k)$ such that $q>1$,
cover $\overline\Omega$ by a system $\P_j$ of nonoverlapping closed cubes contained in $\Omega_j$ and on each 
cube $P\in \P_j$ we perform the construction of Lemma \ref{l:main1}, so that the resulting function
$u_j$  belongs to $W^{2,q}(P)\cap c^{1,\alpha}(P)$ for each $P\in\P$ and satisfies in total
\eqn{uwc}
$$
\|u_j-w_j\|_{C^{1,\alpha}(\Omega)}<2^{-j}
$$
and
\eqn{uww}
$$
 \|u_j-w_j\|_{W^{2,q}(\Omega)}\le C.
$$
Now, by an easy interpolation argument taking into account that the modulus of continuity 
of $\nabla u$ and $\nabla w_j$ is faster than 
$t^{\alpha}$ and  $w_j\to u$ in $C^1(\overline\Omega)$,
we see from \eqref{uwc} that 
$\|u-w_j\|_{C^{1,\alpha}(\Omega)}\to0$. 
Also $w_j-u_j\to 0$ in $C^{1,\alpha}(\Omega)$, so that
$u_j\to u$ in $C^{1,\alpha}(\Omega)$.
Since $W^{2,q}(\Omega)$ is reflexive, we deduce from \eqref{uww}  that
$w_j-u_j\to 0$ weakly in $W^{2,q}(\Omega)$. Since $w_j\to u$ (strongly) in $W^{2,p}(\Omega)$,
together $u_j\to u$ weakly in $W^{2,p}(\Omega)$. It remains to return to the first approximation step, namely, to prove the following.

\begin{lemma}\label{l:lipappr} Let $\Omega\subset\rn$ be a Lipschitz domain, $u\in
W^{2,q}(\Omega)\cap c^{1,\alpha}(\Omega)$ and $\ep>0$. Let $\omega_u$ 
be the modulus of continuity of $\nabla u$.
Then there exists a Lipschitz domain $\Omega'$ such that 
$$
\Omega\subset\subset\Omega'
$$
and a function $v\in W^{2,p}(\Omega')\cap c^{1,\alpha}(\Omega')$ such that 
$$
\|u-v\|_{C^1(\overline\Omega)}<C\ep,
\qquad  \|u-v\|_{W^{2,p}(\Omega)}\le C\ep,
$$
and the modulus of continuity $\omega_v$ of $\nabla v$ satisfies
\eqn{modcon}
$$
\omega_v(t)\le C(\omega_u(t)+t\|u\|_{\infty}).
$$
\end{lemma}

\begin{proof}
Given a vector $z\in \rn$, denote 
$$
\aligned
\Omega_{z}&=\{x\in\rn\colon x-z\in\Omega\},\\
u_{z}(x)&=u(x-z),\qquad x\in\Omega_{z}.
\endaligned
$$
Find $r>0$ such that
$$
\aligned
&z\in\rn,\;|z|<r\implies\\
&\Bigl( \|u-u_{z}\|_{W^{2,p}(\Omega\cap\Omega_{z})}<\ep \quad\&\quad
\|u-u_{z}\|_{C^1(\overline{\Omega\cap\Omega_{z}})}<\ep \Bigr).
\endaligned
$$
Since $\Omega$ is a Lipschitz domain, there exist 
vectors
$z_1,\dots z_m$ such that
$$
|z_i|< r,\qquad i=1,\dots,m
$$
(but $m$ does not depend on $r$!)
and
$$
\overline\Omega\subset\bigcup_{i=1}^m\Omega_{z_i}.
$$
Find
a partition of unity $\sum_{i=1}^m\psi_i$ such that $\psi_i$ are smooth,
$\psi_i\ge 0$,
each $\psi_i$ has its support in $\Omega_{z_i}$ and 
$\sum_i\psi_i=1$ on $\overline\Omega$. 
Set 
$$
v_i=
\begin{cases}
\psi_iu_{z_i},& \text{ on } \Omega_i,\\
0&\text{elsewhere}
\end{cases}
$$
and
$$
v=\sum_{i=1}^mv_i.
$$ 
Then it is easily verified that $v$ has the required properties.
\end{proof}

\bibliography{lmp}
\bibliographystyle{abbrv}

\end{document}